\documentclass[12pt]{article}

\usepackage{amsmath, amssymb, amsthm}
\usepackage{hyperref}

\usepackage{enumerate}
\usepackage{bm}
\usepackage{graphicx}

\usepackage{pgf,tikz}
\usetikzlibrary{arrows}

\usepackage{color}

\usepackage{algorithm, algorithmic}

\newcommand{\N}{\mathbb{N}}
\newcommand{\dist}{\mathrm{d}}

\newcommand{\m}[1]{\mathcal{#1}}

\mathchardef\mh="2D

\newtheorem{thm}{Theorem}[section]

\newtheorem{lem}[thm]{Lemma}
\newtheorem{cor}[thm]{Corollary}
\newtheorem{remark}[thm]{Remark}

\usepackage{authblk}
\begin{document}
\title{On generalized hexagons of order $(3, t)$ and $(4, t)$ containing a subhexagon}
\author[1]{Anurag Bishnoi \thanks{anurag.2357@gmail.com}}
\author[2]{Bart De Bruyn \thanks{bdb@cage.ugent.be}}
\affil[1,2]{Department of Mathematics, Ghent University}

\maketitle

\begin{abstract}
We prove that there are no semi-finite generalized hexagons with $q + 1$ points on each line containing the known generalized hexagons of order $q$ as full subgeometries when $q$ is equal to $3$ or $4$, thus contributing to the existence problem of semi-finite generalized polygons posed by Tits. The case when $q$ is equal to $2$ was treated by us in an earlier work, for which we give an alternate proof. For the split Cayley hexagon of order $4$ we obtain the stronger result that it cannot be contained as a proper full subgeometry in any generalized hexagon.
\end{abstract}

\medskip \noindent {\small \textbf{MSC2010:} 51E12, 05B25}

\section{Introduction}

Generalized polygons were introduced by Jacques Tits in 1959 \cite{Tits} and they are now an integral part of incidence geometry, with connections to several areas of mathematics like group theory, extremal graph theory, algebraic coding theory and design theory. For a given positive integer $n \geq 2$, a generalized $n$-gon can simply be defined as a point-line geometry whose incidence graph has diameter $n$ and girth $2n$. The generalized $3$-gons are equivalent to projective planes, while generalized $4$-gons are precisely the rank $2$ polar spaces \cite{FGQ,Tits2}. By a famous result of Feit and Higman \cite{Feit-Higman}, finite generalized $n$-gons of order $(s, t)$ with $s, t \geq 2$ (the so-called thick ones) exist only for $n \in \{ 2, 3, 4, 6, 8 \}$. In this paper we will mainly be concerned with generalized $6$-gons.

An important open problem in the theory of generalized polygons is the existence of semi-finite generalized polygons, asked by Tits \cite[Appendix E: Problem 5]{VM}. 
These are thick generalized polygons which have finitely many points on each line but infinitely many lines through each point. It is known that semi-finite generalized quadrangles of order $(s,\infty)$ (where $\infty$ is any infinite cardinal number)  do not exist for $s = 2$, $3$ or $4$, as proved by Cameron \cite{Ca}, Brouwer \cite{Br} (independently by Kantor) and Cherlin \cite{Ch}, respectively.  But no such results are known for generalized hexagons or octagons without making any extra assumptions. 
In \cite{ab-bdb:1} we proved among other things that there is no semi-finite generalized hexagon with three points on each line containing a subhexagon of order $(2, 2)$. In this paper we extend this result to semi-finite generalized hexagons of orders $(3, \infty)$ and $(4, \infty)$. 

Every known finite generalized hexagon with an order has its order equal to either $(1, 1)$ (for an ordinary hexagon), $(q, 1)$, $(1, q)$,  $(q, q)$,  $(q, q^3)$ or $(q^3, q)$, where $q$ is a prime power \cite[Chapter 2]{VM}. 
Every generalized hexagon of order $(1,q)$ is isomorphic to the geometry $\mathrm{H}_\pi$ obtained by taking the vertices and edges of the incidence graph of a finite projective plane $\pi$ of order $q$ as points and lines, while every generalized hexagon of order $(q, 1)$ is the point-line dual of such a hexagon. 
If $\pi=\mathrm{PG}(2,q)$, then we also denote $\mathrm{H}_\pi$ by  $\mathrm{H}(1,q)$ and its point-line dual by $\mathrm{H}(q,1)$. The other known generalized hexagons are the split Cayley hexagons of order $q$, denoted by $\mathrm{H}(q)$, the twisted triality hexagons of order $(q^3, q)$, denoted by $\mathrm{T}(q^3, q)$, and their point-line duals. The split Cayley hexagon $\mathrm{H}(q)$ is isomorphic to its dual $\mathrm{H}(q)^D$ if and only if $q$ is a power of $3$ \cite[Cor.~3.5.7]{VM}. 
By a result of Cohen and Tits \cite{CT}, every generalized hexagon of order $(2, 2)$ is isomorphic to the split Cayley hexagon $\mathrm{H}(2)$ or its dual $\mathrm{H}(2)^D$, and every generalized hexagon of order $(2, 8)$ is isomorphic to the dual twisted triality hexagon $\mathrm{T}(8, 2)^D$. 
We also have the following inclusion of geometries: $\mathrm{H}(q, 1) \subseteq \mathrm{H}(q)^D \subseteq \mathrm{T}(q^3, q)^D$ (see \cite{jdk-hvm,VM}). When $q$ is not a power of $3$, i.e., when $\mathrm{H}(q) \not \cong \mathrm{H}(q)^D$, no generalized hexagon is known that contains $\mathrm{H}(q)$ as a full proper subgeometry.\footnote{A subgeometry $\m S$ of a point-line geometry $\m S'$ is called full if for every line $L$ of $\m S$, the set of points incident with $L$ in $\m S$ is equal to the set of points incident with $L$ in $\m S'$.} 
Our main result of this paper is as follows.

\begin{thm} \label{thm:main1}
Let $q \in \{2, 3, 4\}$ and let $\m H$ be a generalized hexagon isomorphic to the split Cayley hexagon $\mathrm{H}(q)$ or its dual $\mathrm{H}(q)^D$. Then the following holds for any generalized hexagon $\mathcal{S}$ that contains $\m H$ as a full subgeometry:
\begin{enumerate}[$(1)$]
\item $\m S$ is finite; 
\item if $q \in \{2, 4\}$ and $\m H \cong \mathrm{H}(q)$, then $\m S = \m H$. 
\end{enumerate}
\end{thm}

\medskip \noindent 
There are some other known results regarding generalized polygons containing certain subpolygons. 
De Medts and Van Maldeghem \cite{DeMedts-VanMaldeghem} proved that $\mathrm{H}(3)$ is the unique generalized hexagon of order $(3, 3)$ containing a subhexagon of order $(3, 1)$. 
A similar characterization of the Ree-Tits octagon of order $(2, 4)$ was given by the second author in \cite{bdb:Ree-Tits} where it was proved that this is the unique generalized octagon of order $(2, 4)$ containing a suboctagon of order $(2, 1)$. 
De Kaey and Van Maldeghem \cite{jdk-hvm} proved the uniqueness of $\mathrm{H}(q)^D$ containing $\mathrm{H}(q, 1)$ by assuming certain extra group theoretical conditions. 

Our proof of Theorem \ref{thm:main1} relies on results about intersection sizes of different ``types'' of hyperplanes in arbitrary finite generalized hexagons of order $(s, t)$ that we obtain in Lemmas \ref{lem:main} and \ref{lem:intersecting_semi-classical}. 
These results are then used along with some computations to deal with the cases where $\m H$ is isomorphic to $\mathrm{H}(2), \mathrm{H}(3)$ or $\mathrm{H}(4)$. 
We hope that these results on intersections of hyperplanes would be interesting in their own right, and useful in obtaining further such results. 
The remaining cases where $\m H$ is isomorphic to $\mathrm{H}(2)^D$ or $\mathrm{H}(4)^D$ are handled by showing that there are no $1$-ovoids (also known as distance-$2$ ovoids) in these geometries. While the non-existence of $1$-ovoids in $\mathrm{H}(2)^D$ is computationally easy to show by an exhaustive computer search, the fact that $\mathrm{H}(4)^D$ has $1365$ points and $1365$ lines makes it quite hard to determine whether $\mathrm{H}(4)^D$ has $1$-ovoids.  We use the result of Bishnoi and Ihringer \cite{ab-fi} that there are no $1$-ovoids in $\mathrm{H}(4)^D$ to finish the proof.
Note that there are no known general results on existence or non-existence of $1$-ovoids in the split Cayley hexagons and their duals, and proving existence or non-existence of $1$-ovoids in these geometries seems like a really hard problem (see \cite{ab-fi} for the current state of our knowledge).
Thus it is natural to try computer-aided proofs in the small cases. 

\section{Basic definitions and properties}  \label{sec:def}

Suppose $\mathcal{S}$ is a point-line geometry with (nonempty) point set $\mathcal{P}$, line set $\mathcal{L}$ and incidence relation $\mathrm{I} \subseteq \m P \times \m L$. The \textit{(point-line) dual} of $\mathcal{S}$ is the point-line geometry $\m S^D = (\m L, \m P, \mathrm{I}^D)$, where $\mathrm{I}^D = \{ (L, x) \, : \, (x, L) \in \mathrm{I} \}$. $\mathcal{S}$ is called a {\em partial linear space} if every two distinct points are incident with at most one common line.

Suppose $\mathcal{S}$ is a partial linear space. A line of $\mathcal{S}$ is called {\em thick} if it contains at least three points. We say that $\mathcal{S}$ has order $(s, t)$ if every line is incident with $s + 1$ points and every point is incident with $t + 1$ lines. If $s = t$, then we simply say that $\mathcal{S}$ has order $s$. The {\em point graph} (or {\em collinearity graph}) of $\m S$ is the graph with vertex set $\m P$ where two points are adjacent whenever they are incident with a common line, i.e., whenever they are collinear. The bipartite graph between points and lines with an edge denoting incidence is called the \textit{incidence graph} of $\mathcal{S}$. In this paper the distance between two points $x$ and $y$ of $\m S$, denoted by $\dist_{\m S}(x, y)$ or simply $\dist(x, y)$, will be the distance between $x$ and $y$ in the point graph of $\m S$. The distance between a point $x \in \m P$ and a nonempty set $X \subseteq \m P$ is defined as the minimum distance between $x$ and a point $y \in X$. We denote the set of points at distance $i$ from a fixed point $x$ by $\Gamma_i(x)$, and similarly the set of points at distance $i$ from a given nonempty set $X$ of points by $\Gamma_i(X)$.  A subset $X$ of points of $\m S$ is called a \textit{subspace} if for every pair of distinct collinear points contained in $X$, all points incident with the unique line joining them are also contained in $X$. The partial linear space $\m S = (\m P, \m L, \mathrm{I})$ is a \textit{subgeometry} of another partial linear space $\m{S}' = (\m{P} ', \m{L}', \mathrm{I} ')$ if $\m{P} \subseteq \m{P}'$, $\m{L} \subseteq \m{L} '$ and $\mathrm{I} = \mathrm{I} ' \cap (\m{P}  \times \m{L})$. The subgeometry is called \textit{full} if for every line $L$ in $\m{L}$ the set $\{x \in \m{P} : x ~\mathrm{I}~ {L}\}$ is equal to $\{x \in \m{P}' : x ~\mathrm{I}'~ L\}$. 

As noted before, a  generalized $n$-gon, for $n \geq 2$, is a point-line geometry whose incidence graph has diameter $n$ and girth $2n$. A generalized polygon is called \textit{thick} if it has at least three points on each line and at least three lines through each point. It can be shown that every (possibly infinite) thick generalized $n$-gon with $n \geq 2$ has an order $(s, t)$ for some fixed (possibly infinite) $s$ and $t$. Clearly, the point-line dual of a generalized $n$-gon of order $(s, t)$ is a generalized $n$-gon of order $(t, s)$. By the Feit-Higman theorem \cite{Feit-Higman}, thick finite generalized $n$-gons exist only for $n \in \{2, 3, 4, 6, 8\}$. The ordinary $n$-gon, which is a generalized $n$-gon of order $(1, 1)$, exists for all $n \geq 3$. All finite non-thick generalized polygons can be obtained from ordinary polygons or thick generalized polygons, see \cite[Thm. 1.6.2]{VM}. For $n = 2$, we have the geometry where every point is incident with every line and for $n = 3$ we have a finite projective plane. Generalized $n$-gons for $n = 4$, $6$ and $8$ are referred to as generalized quadrangles, hexagons and octagons, respectively. Since thick finite generalized $n$-gons for $n > 3$ exist only for even $n$, we can denote them as generalized $2d$-gons, where $n = 2d$ and $d$ is the diameter of the point graph. A \textit{near $2d$-gon} with $d \in \mathbb{N}$ is a partial linear space $\mathcal{S}$ that satisfies the following properties:
\begin{enumerate} [(NP1)]
\item The point graph of $\mathcal{S}$ is connected and has diameter $d$. 
\item For every point $x$ and every line $L$ there exists a unique point $\pi_L(x)$ incident with $L$ that is nearest to $x$. 
\end{enumerate} 
It is well known that a near $2d$-gon, $d \geq 2$, is a generalized $2d$-gon if and only if the following two additional properties are satisfied:
\begin{enumerate}[$(1)$]
\item Every point is incident with at least two lines.
\item For every two points $x$ and $y$ at distance $i \in \{1, 2, \ldots, d - 1\}$ from each other, there exists a unique point collinear with $y$ at distance $i-1$ from $x$, i.e., $|\Gamma_{i - 1}(x) \cap \Gamma_1(y)| = 1$.  
\end{enumerate}

Putting $d = 3$ we get that a point-line geometry is a generalized hexagon if and only if it is a near hexagon in which every pair of points at distance $2$ from each other have a unique common neighbour and every point is incident with at least two lines. 

For the definitions of the split Cayley hexagon $\mathrm{H}(q)$ and the twisted triality hexagon $\mathrm{T}(q^3, q)$ defined over the finite field $\mathbb{F}_q$, we refer to \cite[Chapter~2]{VM}. We shall only need the following facts about generalized hexagons. A finite generalized hexagon of order $(s, t)$ has $(1 + s)(1 + st + s^2t^2)$ points and $(1 + t)(1 + st + s^2t^2)$ lines. The number of points at distance $i$ from a fixed point in such a generalized hexagon is equal to $1, s(t+1), s^2t(t+1), s^3t^2$ for $i = 0,1,2,3$, respectively. 
Let $q$ be a prime power $p^r$, where $p$ is prime and $r$ is a positive integer. Then the automorphism group of $\mathrm{H}(q, 1)$ is isomorphic to $\mathrm{P\Gamma L}_3(q) \rtimes C_2$ and thus it has size $2r(q^3 - 1)(q^3 - q)(q^3  - q^2)/(q - 1)$.
The automorphism group of $\mathrm{H}(q)$ is isomorphic to $\mathrm{G}_2(q) \rtimes \mathrm{Aut}(\mathbb{F}_q)$ and thus it has size $rq^6(q^6 - 1)(q^2 - 1)$. 
The automorphism groups of all known finite thick generalized hexagons act primitively and distance transitively on the points  of the generalized hexagons \cite{Bu-VM}. 

\section{Hyperplanes and valuations} \label{sec:val}

Given a partial linear space $\m S = (\m P, \m L, \mathrm{I})$, a {\em hyperplane} of $\m S$ is a proper subset $H$ of $\m P$ having the property that each line has either one or all its points in $H$.  A \textit{$1$-ovoid} in  $\m S$ is a set $O$ of points having the property that each line of $\m S$ contains a unique point of $O$. Clearly, every $1$-ovoid is a hyperplane. 

Suppose $\m S = (\m P, \m L, \mathrm{I})$ is a generalized $2d$-gon with $d \in \N \setminus \{ 0,1 \}$.  Then a (polygonal)  \textit{valuation} of $\m S$ is a map $f : \m P \rightarrow \mathbb{N}$ that satisfies the following conditions:
\begin{enumerate}[(PV1)]
\item There exists at least one point with $f$-value $0$. 
\item Every line $L$ of $\m S$ contains a unique point $x_L$ such that $f(x) = f(x_L) + 1$ for all points $x \neq x_L$ contained in $L$.  
\item Let $M_f$ denote the maximum value of $f$ over the points of $\m S$.\footnote{It is easy to show that every function satisfying the first two axioms takes values from the set $\{0, \dots, d\}$, and thus has a maximum value.} 
If $x$ is a point with $f(x) < M_f$, then there is at most one line through $x$ containing a (necessarily unique) point $y$ satisfying $f(y) = f(x) - 1$.  
\end{enumerate}

The notion of valuation of a generalized polygon was introduced by the second author in \cite{bdb:polygonal}. It is inspired from a more general notion of valuations of near polygons \cite{bdb-pvdc:1,bdb:valuation,ab-bdb:1}. Valuations have been used to obtain several classification results for near polygons (see \cite{bdb:survey} for a survey).

From property (PV2) it follows that given a valuation $f$ of $\m S$, the set of points with non-maximal $f$-value (i.e. with $f$-value smaller than $M_f$) is a hyperplane of $\m S$, which we denote by $H_f$.  An arbitrary hyperplane $H$ of $\m S$ is said to be of {\em valuation type} if there exists a valuation $f$ of $\m S$ such that $H = H_f$. From the following result it follows that the valuations of a generalized polygon are in bijective correspondence with its hyperplanes of valuation type \footnote{This does not hold true for the more general notion of valuation used in \cite{ab-bdb:1}.} . 

\begin{lem}[{\cite[Prop 3.10]{bdb:polygonal}}]
If $f$ is a valuation of a generalized $2d$-gon $\m S = (\m P, \m L, \mathrm{I})$, then $f(x)=M_f-\dist(x,\m P \setminus \mathcal{H}_f)$ for every point $x$ of $\mathcal{S}$.
\end{lem}

\medskip \noindent Let $\m S = (\mathcal{P},\mathcal{L},\mathrm{I})$ be a generalized $2d$-gon with $d \in \N \setminus \{ 0,1 \}$. Then some examples of hyperplanes of valuation type are as follows \cite[Sec. 3]{bdb:polygonal}:
\begin{enumerate}[$(a)$]
\item Let $p$ be a point of $\m S$. Then the set $H_p = \{ x \in \m P : \dist(x, p) < d\}$ is a hyperplane of $\m S$ and the function $f(x) := \dist(p, x)$ for $x \in \mathcal{P}$ is the valuation of $\mathcal{S}$ corresponding to $H_p$. The hyperplane $H_p$ will be called a {\em singular hyperplane}\footnote{We do not want to use the terminology ``classical hyperplane'' here, since this is often used for a hyperplane that arises from some projective embedding.}, and $f$ will be called a {\em classical valuation}. The point $p$ is called the {\em center} of this hyperplane/valuation.

\item Given a $1$-ovoid $O$, the function $f(x) := 0$ for $x \in O$ and $f(x) := 1$ for $x \not \in O$ defines the valuation of $\m S$ corresponding to the hyperplane $O$. The hyperplane as well as its corresponding valuation will be called {\em ovoidal}.
Conversely, if $f$ is a valuation of $\m S$ with maximum value $1$, then the set of points where $f$ is equal to $0$ form a $1$-ovoid of $\m S$. 

\item Suppose $d \geq 3$. Let $p$ be a point of $\m S$ and let $O'$ be a $1$-ovoid of the subgeometry of $\m S$ induced on the set $\Gamma_d(p)$ by those lines that have distance $d-1$ from $p$. Define $f(x) := \dist(x, p)$ for all points $x$ at distance at most $d - 1$ from $p$, $f(x) := d - 2$ for $x \in O'$, and $f(x) := d - 1$ for the remaining points. Then $f$ is known as a {\em semi-classical valuation} of $\m S$, the hyperplane corresponding to it is known as a {\em semi-singular hyperplane}, and the point $p$ is called the {\em center} of this hyperplane/valuation. When $d$ is equal to $3$, the hyperplane is equal to $\{ p \} \cup \Gamma_1(p) \cup O'$. 
\end{enumerate}

\noindent While every generalized hexagon has singular hyperplanes, determining if it has ovoidal or semi-singular hyperplanes is a difficult problem in general. De Bruyn and Vanhove \cite[Corollary 3.19]{bdb-vanhove} proved that finite generalized hexagons of order $(s, s^3)$, $s > 1$, have no $1$-ovoids. 
In particular, this implies that the dual twisted triality hexagons $\mathrm{T}(q^3, q)^D$ do not have any $1$-ovoids.
For the split Cayley hexagons it was shown by De Wispelaere and Van Maldeghem in \cite{dw-vm:1} that $\mathrm{H}(3)$ has a unique $1$-ovoid, up to isomorphism, and then in \cite{dw-vm:1.5,dw-vm:2} they constructed two non-isomorphic $1$-ovoids of $\mathrm{H}(4)$. Later, Pech and Reichard \cite[Sec.~8.3]{Pech-Reichard} proved using an exhaustive computer search that the two examples constructed by De Wispelaere and Van Maldeghem are the only $1$-ovoids in $\mathrm{H}(4)$, up to isomorphism. In \cite{ab-fi} it has been proved that $\mathrm{H}(4)^D$ has no $1$-ovoids. To our knowledge, the existence of $1$-ovoids in $\mathrm{H}(q)$ and $\mathrm{H}(q)^D$ is not known for any $q \geq 5$. In Section \ref{sec:comp}, we discuss algorithms to compute $1$-ovoids in general point-line geometries. This would help us determine both ovoidal and semi-singular hyperplanes in small generalized hexagons. 

The following lemma shows the importance of valuations of a generalized polygon when studying all generalized polygons containing that generalized polygon as a full subgeometry. 

\begin{lem}[{\cite[Prop.~6.1]{bdb:polygonal}}] \label{lem:embed}
Let $\m S = (\m P, \m L, \mathrm{I})$ be a generalized $2d$-gon contained in a generalized $2d$-gon $\m S' = (\m P', \m L', \mathrm{I'})$ as a full subgeometry. 
Let $x$ be a point of $\m S'$ and put $m := \dist(x, \m P)$. Noting that $m \in \{0, 1, \dots, d - 1\}$, we define $f_x(y) := \dist(x, y) - m$ for every point $y \in \m P$. Then:
\begin{enumerate}[$(1)$]
\item $f_x$ is a valuation of $\m S$ with $M_{f_x} = d - m$. 

\item The valuation $f_x$ is classical if and only if $x$ is a point of $\m S$, semi-classical if and only if $m = 1$ and ovoidal if and only if $m = d - 1$. 

\item If $x_1$ and $x_2$ are two distinct collinear points of $\m S$, then the valuations $f_{x_1}$ and $f_{x_2}$ are distinct. 
\end{enumerate}
\end{lem}
\begin{cor}\label{cor:embed}
Suppose $d = 3$ and $m = 2$ in Lemma \ref{lem:embed}. Then the set of points of $\m S$ where $f_x$ is equal to $0$ is a $1$-ovoid of $\m S$. 
\end{cor}

\medskip \noindent Two lines of a generalized $2d$-gon are called {\em opposite} if they lie at maximal distance $d-1$ from each other. Opposite lines always have the same number of points.

\begin{lem} \label{extra}
Let $\mathcal{S}$ be a generalized $2d$-gon that is contained in a generalized $2d$-gon $\mathcal{S}'$ as a full subgeometry. Then every line $L$ of $\mathcal{S}'$ is opposite to some line of $\mathcal{S}$.
\end{lem}
\begin{proof}
Let $x$ be an arbitrary point of $L$. By Lemma \ref{lem:embed}, we know that there is a point $y$ in $\m S$ at distance $d$ from $x$. This point $y$ has distance $d-1$ from $L$ and so there is a unique line in $\m S'$ through $y$ containing a point of $\Gamma_{d-2}(L)$. Any other line of $\m S'$ through $y$ is opposite to $L$. In particular, there exists a line of $\m S$ through $y$ opposite to $L$.
\end{proof}

\medskip \noindent Let $\m S$ and $\m S'$ be generalized polygons as in Lemma \ref{lem:embed}. For a point $x$ of $\m S'$ the hyperplane of $\m S$ corresponding to the valuation $f_x$ of $\m S$ induced by $x$ will simply be denoted by $H_x$. When $\m S$ and $\m S'$ in Lemma \ref{lem:embed} are generalized hexagons ($d = 3$), then we have $m \leq 2$, and thus the only valuations induced by points of $\m S'$ are classical, semi-classical and ovoidal. Thus, \textit{we only need to study these three types of valuations in a generalized hexagon $\m S$ to study all generalized hexagons that contain $\m S$ as a full subgeometry}. In fact, these are the only types of valuations that can exist in a generalized hexagon \cite[Cor.~3.4]{bdb:polygonal}. 
While the points of $\m S'$ give rise to valuations of $\m S$, the lines of $\m S'$ give rise to certain sets of valuations called \textit{admissible $L$-sets} (see \cite{bdb:polygonal} for a definition),  that we do not discuss in this paper because of their ``technical'' definition. To prove Theorem \ref{thm:main1} we do not need the full machinery of admissible $L$-sets. The following properties of hyperplanes suffice. These properties are implied by \cite[Prop.~4.7]{bdb:polygonal} (which is a result on $L$-sets), but it is possible to give an independent proof as well.

\begin{lem}
\label{lem:hyperplanes}
Let $\m S = (\m P, \m L, \mathrm{I})$ be a generalized $2d$-gon contained in a generalized $2d$-gon $\m S' = (\m P', \m L', \mathrm{I}')$ as a full subgeometry, and let $L$ be a line of $\m S'$. Then:
\begin{enumerate}[$(1)$]
\item the set of hyperplanes $\{H_x : x \in \m P', x~\mathrm{I'}~L\}$ covers $\m P$;
\item if $x_1$, $x_2$ and $x_3$ are three  distinct points on $L$, then $H_{x_1} \cap H_{x_2} = H_{x_1} \cap H_{x_3}$. 
\end{enumerate}
\end{lem}
\begin{proof}
By Lemma \ref{lem:embed} each point $x$ of $\m S'$ determines the hyperplane $H_x$ of $\m S$ defined by taking the points of $\m S$ that are at distance at most $d - 1$ from $x$. 
\begin{enumerate}[$(1)$]
\item By (NP2) for every point $y$ in $\m S$ there exists a unique point $x$ on $L$ nearest to $y$, so that every other point on $L$ is at distance $\dist(x, y) + 1$ from $y$. Since the diameter of $\m S$ is $d$, we must have $\dist(x, y) + 1 \leq d$, implying that $y \in H_{x}$. 
\item By symmetry it suffices to show that $H_{x_1} \cap H_{x_2} \subseteq H_{x_3}$ for any three distinct points $x_1, x_2, x_3$ on $L$.
Let $y \in H_{x_1} \cap H_{x_2}$.
By (NP2) there exists a unique point $y'$ on $L$ nearest to $y$, say at distance $i$, and every other point of $L$ is at distance $i+1$ from $y$.
Since the two distinct points $x_1$ and $x_2$ on $L$ are at distance at most $d-1$ from $y$ we must have $i + 1 \leq d-1$ as at least one of them is distinct from $y'$. 
Therefore, we have $\dist(y, x_3) \leq i+1 \leq d-1$, which is equivalent to $y \in H_{x_3}$. 
\end{enumerate}
\end{proof}

\medskip \noindent A finite generalized hexagon of order $(s, t)$ has $(1+s)(1+st+s^2t^2)$ points. For each of the three types of hyperplanes in generalized hexagons mentioned above, we can determine the sizes of the hyperplanes by simple counting. The singular, semi-singular and ovoidal hyperplanes in a finite generalized hexagon of order $(s, t)$ are of sizes $1 + s(t+1) + s^2t(t+1)$, $1 + s(t+1) + s^2t^2$ and $1 + st + s^2t^2$ respectively.

\begin{lem} \label{lem:main}
Let $\m S$ be a finite generalized hexagon of order $(s, t)$ contained in a generalized hexagon $\m S'$ as a full subgeometry, and let $L$ be a line of $\m S'$ that does not intersect $\m S$. Let $n_L$ denote the number of points on $L$ that are at distance $2$ from $\m S$. 
For a point $x$ in $\m S'$, let $H_x$ denote the hyperplane of $\m S$ formed by taking points of $\m S$ that are at non-maximal distance from $x$. 
Then for any two distinct points $x$ and $y$ on $L$, the cardinality of $H_x \cap H_y$ is  equal to $s + 1 -  n_L$.
\end{lem}
\begin{proof}
By Lemma \ref{extra} and the fact that opposite lines are incident with the same number of points, we know that every line of $\m S'$ is incident with precisely $s+1$ points. By Lemma \ref{lem:embed}, for every point $x$ on $L$ the hyperplane $H_x$ of $\m S$ is either semi-singular or ovoidal. 
The points on $L$ that are at distance $1$ from $\m S$ induce semi-singular hyperplanes, while those at distance $2$ induce ovoidal hyperplanes. 
By Lemma \ref{lem:hyperplanes} there exists a fixed subset $X$ of points of $\m S$ such that $H_x \cap H_y = X$ for every pair of distinct points $x$, $y$ on $L$, and every point of $\m S$ is contained in some hyperplane induced by a point on $L$. Let the size of $X$ be $k$. There are $n_L$ hyperplanes of size $1+ st + s^2t^2$ (ovoidal) and $s + 1 - n_L$  hyperplanes of size $1 + s(t+1) + s^2t^2$ (semi-singular) which cover a set of size $(1+s)(1+st+s^2t^2)$ (points of $\m S$) and pairwise intersect in $k$ points. Therefore, we have 
\[ n_L (1 + st + s^2t^2 - k) + (s + 1 - n_L)(1 + s(t + 1) + s^2t^2 - k) + k = (1 + s)(1 + st + s^2t^2), \]
which can be solved for $k$ to get $k = s+ 1 - n_L$. 
\end{proof}

\begin{lem} \label{lem:intersecting_semi-classical}
Let $\m S$ be a finite generalized hexagon of order $(s, t)$ having the property that $|H_1 \cap H_2| > s+1$ for any two semi-singular hyperplanes $H_1$ and $H_2$ of $\mathcal{S}$ whose centers lie at distance $3$ from each other. Then there does not exist any generalized hexagon that contains $\m S$ as a full proper subgeometry. 
\end{lem}
\begin{proof}
Suppose that $\m S'$ is such a generalized hexagon and let $x$ be a point of $\m S'$ that is at distance $1$ from the point set of $\m S$.  By Lemma \ref{lem:embed}, $H_x$ is a semi-singular hyperplane, corresponding to the semi-classical valuation $f_x$ defined by $f_x(y) = \dist(x, y) - 1$ for points $y$ of $\mathcal{S}$. Let $x'$ be the unique point of $\m S$ with $f_x$-value $0$, or equivalently the unique point of $\m S$ at distance $1$ from $x$. Let $O'$ be the $1$-ovoid in the subgeometry of $\m S$ induced on $\Gamma_3(x')$ which defines the hyperplane $H_x$. Let $y$ be a point of $O'$. Then $f_x(y) = 1$, and hence $\dist(x, y) = 2$. 
Let $z$ be the common neighbour of $x$ and $y$. Then $z$ must lie outside $\m S$, and the line $L = xz$ does not contain any point of $\m S$. Note that $f_z$ is also a semi-classical valuation since $z$ has distance 1 from $\mathcal{S}$. From Lemma \ref{lem:main} it follows that $|H_x \cap H_z| \leq s + 1$. Moreover we have $\dist(x', y) = 3$, thus contradicting the assumption stated in the lemma. 
\end{proof}

We will use Lemma \ref{lem:intersecting_semi-classical} to prove Theorem \ref{thm:main1}$(2)$ in Section \ref{sec:proof1}.
One way to prove the finiteness of generalized hexagons containing a subhexagon is as follows, which we use to prove the case when $\m H$ is isomorphic to $\mathrm{H}(2)^D$, $\mathrm{H}(3)$ or $\mathrm{H}(4)^D$ in Theorem \ref{thm:main1}. 

\begin{lem}
\label{lem:main2}
Let $\m S$ be a finite generalized hexagon with only thick lines that is contained in a generalized hexagon $\m S'$ as a full subgeometry. If every point of $\m S'$ is at distance at most $1$ from $\m S$, then $\m S'$ is also finite. 
\end{lem}
\begin{proof}
As opposite lines have the same number of points, Lemma \ref{extra} implies that every line of $\m S'$ is also thick. In any generalized hexagon, opposite points (i.e. points at distance 3) are incident with the same number of lines. The fact that every line of $\m S'$ is thick then implies that the points of $\m S'$ are incident with a constant number of lines. 
Indeed, since the lines of $\m S'$ are thick,  for any two distinct collinear points $u$ and $v$ of $\m S'$, there is a point $w$ opposite to both $u$ and $v$, see \cite[\S 1.5]{VM}.

Suppose now that every point of $\m S'$ is at distance at most $1$ from $\m S$ and that $\m S' \neq \m S$. Let $x$ be a point in $\m S'$ at distance $1$ from $\m S$. Then it suffices to show that there are only finitely many lines through $x$.

Note that $x$ induces a semi-classical valuation $f_x$ on $\m S$, and thus there exists a unique point $y$ of $\m S$ with $f_x$-value $0$, which by Lemma \ref{lem:embed} is the unique point of $\m S$ at distance $1$ from $x$. Therefore, there is a unique line through $x$ in $\m S'$ which meets $\m S$. Now, let $L$ be any other line through $x$. Pick any point $z$ in $L \setminus \{ x \}$. Then $z$ is again collinear with a unique point $z'$ of $\m S$ as every point of $\m S'$, and in particular $z$, is at distance at most $1$ from $\m S$. In this manner we can correspond each line of $\m S'$ through $x$ that does not intersect $\m S$ to a point of $\m S$. Moreover, for two distinct lines $L_1, L_2$ through $x$ not meeting $\mathcal{S}$ the points $z_1',z_2'$ of $\m S$ obtained in this manner by taking points $z_1 \in L_1 \setminus \{x\}$ and $z_2 \in L_2 \setminus \{x\}$ must be distinct, as otherwise we will get a pair of points at distance $2$ from each in the generalized hexagon $\m S'$ that have at least two common neighbours. Since the number of points in $\m S$ is finite, this shows that there are only finitely many lines through $x$. 
\end{proof}

\begin{cor}
\label{cor:ovoids}
If a generalized hexagon $\m S$ does not have any $1$-ovoids, then it cannot be contained in a semi-finite generalized hexagon as a full subgeometry. 
\end{cor}
\begin{proof}
Let $\m S'$ be a generalized hexagon containing $\m S$ as a full subgeometry. Suppose that $\m S$ has no $1$-ovoids. Then from Corollary \ref{cor:embed} it follows that every point of $\m S'$ is at distance at most $1$ from $\m S$, and so $\m S'$ must be finite by Lemma \ref{lem:main2}. 
\end{proof}

\section{Proof of Theorem \ref{thm:main1}} \label{sec:comp} \label{sec:proof1}

Let $\m S = (\m P, \m L, \mathrm{I})$ be a point-line geometry. If for every pair of distinct lines $L_1, L_2 \in \m L$ we have $\{x \in \m P : x ~\mathrm{I}~L_1\} \neq \{x \in \m P : x ~\mathrm{I}~L_2\}$ then every line of $\m S$ can be uniquely identified with the set of points incident with the line. This condition holds true for partial linear spaces, and hence for near polygons and generalized $n$-gons with $n \geq 3$. Thus we can look at these point-line geometries as hypergraphs $(V, E)$ where $V = \m P$ and $E = \{\{x \in \m P : x~\mathrm{I}~L\} : L \in \m L\}$. A $1$-ovoid in $\m S$ is then equivalent to an exact hitting set in the hypergraph $(V, E)$. 
Determining whether an arbitrary hypergraph contains an exact hitting set, which is equivalent to determining if the dual hypergraph has an exact cover,  is a well known NP-hard problem \cite{Ka}.

One of the most famous algorithms to determine exact covers in hypergraphs is the Dancing Links algorithm by Knuth \cite{Kn}. We will use this algorithm to compute both ovoidal and semi-singular hyperplanes in small generalized hexagons. The algorithm is already implemented in SAGE \cite{sage} under the name of \verb|DLXCPP| and its code is publicly available \footnote{see \url{http://www.sagenb.org/src/combinat/matrices/dlxcpp.py}}. The following function written in SAGE computes all $1$-ovoids in a given point-line geometry. 

\begin{verbatim}
def ovoids(P,L):
    """
    Find all 1-ovoids in a point-line geometry. 

    Args:
    P -- the points of the geometry
    L -- the lines of the geometry

    Returns:
    a list of 1-ovoids (exact hitting sets) in the geometry
    """
    map = dict() # to construct the dual problem of exact covers
    for p in P:
        map[p] = []
    for i in range(len(L)):
        for p in L[i]:
            map[p].append(i)
    E = [map[p] for p in P]
    return list(DLXCPP(E))
\end{verbatim}
This function can also be used to compute semi-singular hyperplanes of a generalized polygon, since they correspond to $1$-ovoids in the subgeometry induced on the points at maximum distance from a given point.

Computer models of the point-line geometries $\mathrm{H}(2), \mathrm{H}(2)^D, \mathrm{H}(3), \mathrm{H}(4)$ and $\mathrm{H}(4)^D$ can easily be constructed in GAP \cite{gap} using the fact that the automorphism groups of these generalized hexagons act primitively on the set of points; the function \verb|AllPrimitiveGroups| can be used to obtain all primitive groups of given degree and size \footnote{Another way of constructing these geometries in GAP is via the FinIng package which is available in the most recent version of GAP.}. The file \verb|main.g| in \cite{code} contains the code we used to construct the automorphism groups, points and lines of these geometries. This data can alternatively be obtained from the online database on small generalized polygons maintained by Moorhouse \cite{Moorhouse}. 
The function \verb|ovoids| when run on points and lines of $\mathrm{H}(2)^D$ immediately shows that this generalized hexagon has no $1$-ovoids, and thus by Corollary \ref{cor:ovoids} it cannot be contained as a full subgeometry in a semi-finite generalized hexagon. The same conclusion holds for $\mathrm{H}(4)^D$ by the results of \cite{ab-fi}. The main idea behind the computations in \cite{ab-fi} is the observation that since $\mathrm{H}(4, 1)$ is a subgeometry of $\mathrm{H}(4)^D$, we can first classify all $1$-ovoids of $\mathrm{H}(4, 1)$, which correspond to perfect matchings in the incidence graph of $\mathrm{PG}(2, 4)$, up to isomorphism  under the action of the stabilizer, and then see if any of these $1$-ovoids ``extends'' to a $1$-ovoid of $\mathrm{H}(4)^D$.

Let $\m H \cong \mathrm{H}(3)$ and let $\m S$ be a generalized hexagon containing $\mathrm{H}(3)$ as a full subgeometry. Say there exists a point $x$ in $\m S$ at distance $2$ from $\m H$ and let $x, y, z$ be a path of length $2$ from $x$ to a point $z$ of $\m H$. Then by Lemma \ref{lem:main} we have $|H_x \cap H_y| \leq 3$, where $H_x$ is the $1$-ovoid of $\m H$ induced by $x$ and $H_y$ is the semi-singular hyperplane of $\m H$ induced by $y$. Note that $z \in H_x$ and $z$ is the center of $H_y$. Since the geometry is small enough and the automorphism group acts transitively on the points, we can fix a point $p$ of $\mathrm{H}(3)$ in our computer model and look at the intersection sizes of 1-ovoids through $p$ and semi-singular hyperplanes with center $p$. The code in the file \verb|main.sage| \cite{code} checks these intersection sizes and we have found that every pair of ovoidal and semi-singular hyperplane of $\mathrm{H}(3)$ through a fixed point which is moreover the center of the semi-singular hyperplane intersect in more than $3$ points. Therefore, every point of $\m S$ must be at distance at most $1$ from $\m H$. Then it follows from Lemma \ref{lem:main2} that $\m S$ is finite. 

Finally, let $\m H$ be isomorphic to $\mathrm{H}(2)$ or $\mathrm{H}(4)$ and let $q$ be the order of $\m H$. By Lemma \ref{lem:intersecting_semi-classical}, to show that $\m H$ cannot be embedded in any generalized hexagon as a proper full subgeometry, it suffices to check that for every pair of points $x_1, x_2 \in \m H$ at distance $3$ from each other and for every pair of semi-singular hyperplanes $H_1, H_2$ with respective centers $x_1$ and $x_2$, we have $|H_1 \cap H_2| > q+1$. 
We have done this check in \verb|main.sage| \cite{code}. Note that by distance transitivity of the automorphism group we only need to check this for one pair of points at distance $3$ from each other, thus reducing the amount of computations. 

\begin{remark} {\em Generalized hexagons of order greater than $4$ seem to be out of reach with our computational methods. And we do not know of any results on intersection sizes of semi-singular and ovoidal hyperplanes of split Cayley hexagons that can help us obtain the above results in general. It would be nice to be able to prove that for all prime powers $q = p^r$, with $p \neq 3$ prime, every pair of semi-singular hyperplanes in $\mathrm{H}(q)$ whose centers are at maximum distance $3$ intersect in more than $q + 1$ points, which will then imply that these generalized hexagons cannot be contained in bigger generalized hexagons as full subgeometries.}
\end{remark}

\section{Further questions}

\begin{enumerate}[$(1)$]
\item For any prime power $q > 4$, is there any semi-finite generalized hexagon that contains either the split Cayley hexagon $\mathrm{H}(q)$ or its dual $\mathrm{H}(q)^D$ as a subgeometry? We conjecture that there are no such semi-finite generalized hexagons. 

\item For any prime power $q$, is there a semi-finite generalized hexagon that contains $\mathrm{H}(q,1)$ as a subgeometry? We believe that this problem is much harder than $(1)$ as none of our techniques have worked so far in solving it (not even for the smallest case $q=2$).

\item Is the dual  split Cayley hexagon $\mathrm{H}(q)^D$ the unique generalized hexagon of order $(q, q)$ containing the hexagon $\mathrm{H}(q, 1)$ as a subgeometry? 
For $q = 3$ this was proved by De Medts and Van Maldeghem but it remains open for all $q > 3$. 

\item Is the dual twisted triality hexagon $\mathrm{T}(q^3, q)^D$ the unique generalized hexagon of order $(q, q^3)$ containing the dual split Cayley hexagon $\mathrm{H}(q)^D$ as a subgeometry?
The theory of valuations can be useful in proving this for $q = 3$, and in fact we have used it to derive some  properties of generalized hexagons of order $(3, t)$ containing $\mathrm{H}(3) \cong \mathrm{H}(3)^D$, but so far we have not been successful.

\end{enumerate}

\end{document}